\documentclass[12pt,a4paper]{article}
\usepackage{mathrsfs}\usepackage{color}
\usepackage{amsfonts}
\usepackage{amssymb}
\usepackage{amsthm}
\usepackage{amsmath}
\usepackage{latexsym}
\begin{document}
\newcommand{\rco}{\color{red}}
\newtheorem{thm}{Theorem}[section]
\newtheorem{dei}{Definition}
\newtheorem{lem}{Lemma}
\newtheorem{exa}{Example}
\newtheorem{cor}{Corollary}
\newtheorem{pro}{Proposition}
\newtheorem{ass}{Assumption}
\theoremstyle{plain}
\newtheorem{rmk}{Remark}
\newcommand{\sectemul}{\arabic{section}}
\renewcommand{\theequation}{\sectemul.\arabic{equation}}
\renewcommand{\thepro}{\sectemul.\arabic{pro}}
\renewcommand{\thelem}{\sectemul.\arabic{lem}}
\renewcommand{\thethm}{\sectemul.\arabic{thm}}
\renewcommand{\thecor}{\sectemul.\arabic{cor}}
\renewcommand{\theexa}{\sectemul.\arabic{exa}}
\renewcommand{\thedei}{\sectemul.\arabic{dei}}
\renewcommand{\theass}{\Alph{ass}}
\renewcommand{\thermk}{\sectemul.\arabic{rmk}}
\renewcommand{\thesubsection}{\thesection.\arabic{subsection}}
\newcommand{\nn}{\nonumber}
\newcommand{\ol}{\overline}
\newcommand{\be}{\begin{equation}}
\newcommand{\ee}{\end{equation}}
\newcommand{\ben}{\begin{equation*}}
\newcommand{\een}{\end{equation*}}
\newcommand{\ba}{\begin{array}}
\newcommand{\ea}{\end{array}}
\newcommand{\beqn}{\begin{eqnarray*}}
\newcommand{\eeqn}{\end{eqnarray*}}
\newcommand{\beq}{\begin{eqnarray}}
\newcommand{\eeq}{\end{eqnarray}}
\newcommand{\bl}{\begin{Lemma}}
\newcommand{\el}{\end{Lemma}}
\newcommand{\bt}{\begin{Theorem}}
\newcommand{\et}{\end{Theorem}}
\newcommand{\p}{{\heiti proof}\quad}
\newcommand{\ep}{eproof.}


\title{ Precise local large deviations for random sums with applications to risk models%
\thanks {This work was supported by National Natural Science Foundation of China (No. 11401415) }}
\author{\small Qiuying Zhang, Fengyang Cheng\thanks{Corresponding author. E-mail:
 chengfy@suda.edu.cn}\\
\small Department of Mathematics, Soochow University, Suzhou, 215006, China\\
}
\date{}
\maketitle
{\bf ABSTRACT. }
In this paper, we investigate the precise local large deviation probabilities for random sums of independent real-valued random variables with a common distribution $F$, where $F(x+\Delta)=F((x, x+T])$ is an $\mathcal{O}$-regularly varying function for some fixed constant $T>0$(finite or infinite). We also obtain some results on precise local large deviation probabilities for the claim surplus process of generalized risk models in which the premium income until time $t$ is simply assumed to be a nondecreasing and nonnegative stochastic process. In particular, the results we obtained are also valid for the global case, i.e. case $T=\infty$.


\section{Introduction}\setcounter{equation}{0}\setcounter{dei}{0}\setcounter{exa}{0}
\setcounter{thm}{0}\setcounter{cor}{0}\setcounter{lem}{0}\setcounter{pro}{0}\setcounter{rmk}{0}
\mbox{}

Throughout this paper, let $T$ be a positive constant or $\infty$, and  denote $\Delta=\Delta(T)=(0,T]$ if $T<\infty$ and $\Delta=\Delta(T)=(0,\infty)$ if $T=\infty$. In addition, for any real $x$, we write $x+\Delta=\{z\in \mathbb{R}: z=x+y,y\in\Delta\}$.

Let $\{X,X_k:k\geq1\}$ be a sequence of independent and identically distributed (i.i.d) random variables (r.v.s) with a common distribution $F$ and a finite mean $\mu=EX $, and let $\{ N(t),t \geq 0\}$ be a counting process with a finite mean function $\lambda (t)=EN(t)$ which tends to $\infty$ as $t \to \infty $. Furthermore, suppose that $\{X,X_k:k\geq1\}$ and  $\{ N(t),t \geq 0\}$ are independent.

In this paper, we will investigate the precise local large deviation probabilities  of random sums
\ben S_{N(t)}=\sum\limits_{k=1}^{N(t)}X_k,~t\geq 0,~\sum_{k=1}^0X_k=0,\een
which states that under some suitable conditions, for every fixed $\gamma>0$, the relation
{\small\begin{equation}\label{eq100}
C_1\leq\lim_{t\to\infty}\inf_{x\geq \gamma \lambda(t)}\frac{P(S_{N(t)}-\mu \lambda (t)\in x+\Delta)}{\lambda(t)F(x+\Delta)}\leq\lim_{t\to\infty}\sup_{x\geq \gamma \lambda(t)}\frac{P(S_{N(t)}-\mu \lambda (t)\in x+\Delta)}{\lambda(t)F(x+\Delta)}\leq C_2
\end{equation}}
holds, where $C_1$ and $C_2$ are two indices of the function $F(x+\Delta)$.


When $T=\infty$(it is called the global case) and $C_1=C_2=1$, relation (\ref{eq100}) has been investigated by many researchers, see 
  Kl\"{u}ppelberg and Mikosch (1997), 
   Embrechts et al. (1997),  Mikosch and Nagaev (1998) and Tang et al. (2001), among many others.
 Recent advances on precise global large deviation probabilities for random sums and risk models can be found in Chen and Zhang (2007), Chen et al. (2011), Chen et al. (2014), Ng et al. (2003), Ng et al. (2004),  Wang and Wang (2013) and references therein.


However, due to the local case (i.e. $T<\infty$) is difficult to handle, results for the precise local large deviation probabilities  of random sums are rare.
But the local case 
 is also very useful in practical applications.
 For example, consider the sales status of a commodity in a large supermarket: Let $X_k$ denote the amount of the commodity by the $k-$th consumer,  
 $k=1,2,\cdots$ and let $N(t)$ denote the total number of the customers 
  until time $t$, $t\geq 0$. Then, the random sums $S_{N(t)}=\sum_{k=1}^{N(t)}X_k$ represent total sales of the commodity in the supermarket until time $t$. In practical applications,  the supplier usually give appropriate discount to supermarkets based on their sales volume range. Therefore, for some positive numbers $x$ and $T$, a good estimation of $P(S_{N(t)}\in x+\Delta)$ is obviously helpful for the supermarket developing pricing strategy.


The second goal of this paper is to investigate the precise local large deviation for the claim surplus process of a generalized risk model, in which the premium income until time $t$ is simply assumed to be a nondecreasing and nonnegative stochastic process. The model can be described as follows:
 \begin{itemize} \setlength{\itemsep}{-18pt}
 \item The claim number until time $t$ is a counting process $N(t)$ with a mean function $EN(t) = \lambda (t)$ which tends to $\infty$ as $t \to \infty $; \\
\item  The premium income until time $t$ is a nondecreasing and nonnegative stochastic process $\{ Y(t),t \geq 0\}$ with a mean function $EY(t) =  b (t)$  which tends to $\infty$ as $t \to \infty $ also;\\
\item The individual claim sizes $\{ X_k,k \geq 1\} $ are i.i.d nonnegative r.v.s with a common distribution $F$ and a finite mean $\mu $;\\
\item  In addition, assume that $\{Y(t),t\geq 0\}$, $\{N(t),t\geq 0\}$ and $\{X_k,k\geq 1\}$ are mutually independent.
\end{itemize}
Suppose that $x>0$ is the initial reserve of a insurance company, then the risk reserve process is given by
$ R(t)=x+Y(t)-\sum\limits_{i = 1}^{N(t)} {{X_i}}$
and the claim surplus process is denoted by
$S(t) = \sum\limits_{i = 1}^{N(t)} {{X_i}}  - Y(t).$
We will prove that, under some suitable conditions, for every fixed $\gamma>\nu$, where $\nu$ is a positive constant, the relation
\be\label{eq1.2}
C_1\leq\lim_{t\to\infty}\inf_{x\geq \gamma \lambda(t)}\frac{P(S(t)-ES(t)\in x+\Delta)}{\lambda(t)F(x+\Delta)}\leq\lim_{t\to\infty}\sup_{x\geq \gamma \lambda(t)}\frac{P(S(t)-ES(t)\in x+\Delta)}{\lambda(t)F(x+\Delta)}\leq C_2
\end{equation}
holds for some positive constants $C_1$ and $C_2$.

 We particularly point out that when $T=\infty$ and $C_1=C_2=1$, (\ref{eq1.2}) reduces to
\begin{equation*}\label{eq104}
   \lim_{t\to\infty}\sup_{x\geq \gamma \lambda(t)}\left|\frac{P(S(t)-ES(t)> x )}{\lambda(t)\overline{F}(x )}-1\right|=0,
\end{equation*}
which was investigated by many researchers such as Tang et al. (2001), Hu (2004), Chen et al. (2011), Chen et al. (2014), Ng et al. (2004), and references therein.

The rest of this paper consists of three sections. Section 2 gives some notations and introduces several function classes. Section 3 presents the main results. Proofs of theorems and corollaries are arranged in Section 4.

\section{Preliminaries}
\setcounter{equation}{0}\setcounter{dei}{0}\setcounter{exa}{0}
\setcounter{thm}{0}\setcounter{cor}{0}\setcounter{lem}{0}\setcounter{pro}{0}\setcounter{rmk}{0}\

First, we introduce some notations and notions which will be valid in the rest of this paper.
Let $a(x)$ and $b(x)$ be two nonnegative unary functions. We write $a(x) \sim b(x)$ if $\mathop {\lim }\limits_{x \to \infty } \frac{{a(x)}}{{b(x)}} = 1$, we write $a(x) =o(b(x))$ if $\lim\limits_{x \to \infty } \frac{a(x)}{b(x)}=0$ and we write $a(x) =O( b(x))$ if $\limsup\limits_{x \to \infty } \frac{a(x)}{b(x)}<\infty$.

Let $a(t,x),~ b(t,x)$ be two nonnegative binary functions and let ${{\cal D}_t}\neq \emptyset $ be some $x$-region. We say that
$a(t,x) \lesssim b(t,x)$ (or equivalently, $b(t,x) \gtrsim a(t,x)$)
 holds uniformly for all $x \in {\cal D}_t $ as $t \to \infty $ if
  $\mathop {\lim }\limits_{t \to \infty } \mathop {\sup }\limits_{x \in {{\cal D}_t}} \frac{{a(t,x)}}{{b(t,x)}} \leq 1$,
and we say that $a(x,t)\sim b(x,t)$  holds uniformly for all $x \in {\cal D}_t $ if both 
$a(t,x) \lesssim b(t,x) \textup{~ and~} b(t,x) \lesssim a(t,x)$
 hold uniformly for all $x \in {\cal D}_t $ as $t \to \infty $.
 Furthermore, we say that $a(x,t)=o(b(x,t))$  holds uniformly for all $x \in {\cal D}_t $ if 
  $\mathop {\lim }\limits_{t \to \infty } \mathop {\sup }\limits_{x \in {{\cal D}_t}} \frac{{a(t,x)}}{{b(t,x)}} =0.$

Next, we introduce some function classes which will be used in this paper.
\begin{dei}\label{dei>ir}
Let $f$ be an eventually positive function, i.e. $f(x)>0$ for all sufficiently large $x$. $f$ is called intermediate regularly varying, denoted by $f \in \mathcal{IR}$,  if
$\mathop {\lim }\limits_{y \downarrow 1} \mathop {\liminf}\limits_{x \to \infty } \frac{{f(xy)}}{{f(x)}} = \mathop {\lim }\limits_{y \downarrow 1} \mathop {\limsup }\limits_{x \to \infty } \frac{{f(xy)}}{{f(x)}} = 1;$
$f$ is called $\mathcal{O}$-regularly varying, denoted by $f \in \mathcal{OR}$,  if
$0 < \mathop {\liminf}\limits_{x \to \infty } \frac{{f(xy)}}{{f(x)}} \leq \mathop {\limsup }\limits_{x \to \infty } \frac{{f(xy)}}{{f(x)}} < \infty$
holds for each fixed $y \geq 1$;
$f$ is said to be long tailed, denoted by $f \in \mathcal{L}$, if 
 $\lim\limits_{x\to\infty}\frac{f(x+y)}{f(x)}=1$
for every fixed $y\in (-\infty,\infty)$.
\end{dei}

By Corollary 1.2 in Cline (1994), 
 ~it is well known that $\mathcal{IR} \subset \mathcal{L}\cap\mathcal{OR}$. Furthermore, we remark that, if $\overline{F}(x):=F(x,\infty)\in \mathcal{IR}$, then $F$ is called to belong to distribution class $\mathcal{C}$, and if $\overline{F}\in \mathcal{OR}$, then $F$ is called to belong to distribution class $\mathcal{D}$.

For an eventually positive function $f$, we introduce some function indices as follows:
\begin{align*}
\alpha (f) &= \lim _{y \to \infty } \frac{\log (\limsup\limits_{x \to \infty } \frac{f(xy)}{f(x)})}{\log y} ;\\
\beta (f)  &= \lim_{y \to \infty } \frac{\log (\liminf\limits_{x \to \infty } \frac{f(xy)}{f(x)})}{\log y};\\
 l_f&=\lim\limits_{\epsilon \downarrow0}\liminf\limits_{x\to\infty}\frac{\inf\limits_{(1-\epsilon )x\leq z\leq(1+\epsilon )x}f(z)}
  {f(x)};\\
  L_f&=\lim\limits_{\epsilon \downarrow0}\limsup\limits_{x\to\infty}\frac{\sup\limits_{(1-\epsilon )x\leq z\leq(1+\epsilon )x}f(z)}
  {f(x)}.
\end{align*}
We call $\alpha (f)$ and $\beta (f)$ the upper and lower Matuszewska's indices of $f$, respectively. 
By Corollary 2.2I in Cline (1994), it is obviously that $f \in \mathcal{IR}$ iff $l_f =L_f.$

The following proposition can be found in Bingham et al. (1987):
\begin{pro}\label{pro2.2}
Suppose that $f$ is an eventually positive function.

\textup{(i)} 
  $f \in \mathcal{OR}$ if and only if both $\alpha (f)$ and $\beta (f)$ are finite.

\textup{(ii)} 
 If $f \in \mathcal{OR}$, then for every $\alpha  > \alpha (f)$, there exist positive numbers ${c_\alpha }$ and ${x_\alpha }$ such that
\begin{equation}\label{eq2.4}
  \frac{{f(xy)}}{{f(x)}} \leq {c_\alpha }{y^\alpha }
\end{equation}
 holds for each $y\geq 1$ and all $x\geq{x_\alpha }$. Similarly, for every $\beta  < \beta (f)$, there exist positive numbers ${c_\beta }$ and ${x_\beta }$ such that
 \begin{equation}\label{eq2.5}
  \frac{{f(xy)}}{{f(x)}} \geq {c_\beta }{y^\beta }
 \end{equation}
 holds for each $y\geq 1$ and all $x\geq{x_\beta }$.\\
\end{pro}





Finally, we will end this section by providing the definition of almost decreasing function, which was introduced by Aljan\v{c}i\'{c} and Aran{d}elovi\'{c} (1977).

\begin{dei} An eventually positive function $f $ is said to be almost decreasing if
$$\limsup\limits_{x\to\infty}\frac{\sup\limits_{u\geq x}f(u)}{f(x)}<\infty.$$
\end{dei}

\newpage

\section{Main results}
\setcounter{equation}{0}\setcounter{dei}{0}\setcounter{exa}{0}
\setcounter{thm}{0}\setcounter{cor}{0}\setcounter{lem}{0}\setcounter{pro}{0}\setcounter{rmk}{0}\mbox{}

In this section, we will present the main results of this paper. The proofs of theorems and corollaries are arranged in section 4.

The first theorem gives the precise local large deviation probabilities for random sums:
\begin{thm}\label{thm>thm31}
Let $\{ X,X_k,k \geq 1\} $ be a sequence of $i.i.d$ real valued $r.v.s$ with a common distribution $F$ of finite mean $\mu $. Let $\{ N(t),t \geq 0\} $ be a counting process with finite mean function $\lambda (t)=EN(t)$ which tends to $\infty$ as $t \to \infty $. Assume that $F_{\Delta}(x)=F(x + \Delta )$ is almost decreasing and $E(X^+)^r < \infty $ for some $r > 1$, where $X^+=XI(X\geq 0)$ and  $I(A)$ is the indicator function of the set $A$.

 Furthermore, suppose that one of the following two conditions holds:

\textup{(i)} $\mu \geq 0$ and the relation
\begin{equation}\label{T3101}
  EN^p(t)I(N(t) > (1 + \delta )\lambda (t)) = O(\lambda (t)),~t\to\infty
\end{equation}
holds for some $p > |\beta ({F_\Delta })|$ and all $\delta  > 0$;

\textup{(ii)} $\mu < 0$ and the relation
\begin{equation}\label{T3102}
  P(N(t) \leq (1 - \delta )\lambda (t)) =  o (\lambda (t)F(\lambda (t) + \Delta )),~~t\to\infty
\end{equation}
holds for all $0 < \delta  < 1$. \\
If $F(x + \Delta ) \in \mathcal{OR}$, then for every fixed $\gamma  > 0$, the relation
\begin{equation}\label{T3103}
  l_{F_\Delta}^2\lambda (t) F(x + \mu  + \Delta ) \lesssim P({S_{N(t)}} - \mu \lambda (t) \in x + \Delta ) \lesssim L_{F_\Delta}^2\lambda (t) F(x + \mu  + \Delta )
\end{equation}
holds uniformly for all $x \geq \gamma \lambda (t)$ as $t \to \infty $.
In particular, if $F(x + \Delta ) \in \mathcal{IR}$, then for every fixed $\gamma  > 0$, the relation 
\begin{equation}\label{eq101}
 P({S_{N(t)}} - \mu \lambda (t) \in x + \Delta ) \sim \lambda (t)F(x + \mu+\Delta )
\end{equation}
holds uniformly for all $x \geq \gamma \lambda (t)$ as $t \to \infty $.
\end{thm}

From Theorem \ref{thm>thm31}, we can easily obtain the following results.
\begin{cor}\label{cor>cor31}
Let $\{ X,X_k,k \geq 1\} $ be a sequence of $i.i.d$ real valued $r.v.s$ with a common distribution $F$ of finite mean $\mu $, and let $\{ N(t),t \geq 0\} $ be a counting process with finite mean function $\lambda (t)=EN(t)$ which tends to $\infty$ as $t \to \infty $. Suppose that $F_{\Delta}(x)=F(x + \Delta )$ is almost decreasing and $E{(X^+ )^r} < \infty $ for some $r > 1$.  For a real number c, assume that one of the following two conditions holds for a real number $c$:

\textup{(i)}  $\mu +c \geq 0$ and the relation (\ref{T3101}) holds for some $p > |\beta ({F_\Delta })|$ and all $\delta  > 0$;

\textup{(ii)} $\mu +c < 0$ and the relation (\ref{T3102}) holds for all $0 < \delta  < 1$.\\
If $F(x + \Delta ) \in \mathcal{L}\cap\mathcal{OR}$, then for every fixed $\gamma  > c$, the relation
\beq 
  l_{F_\Delta}^2\lambda (t) F(x - c\lambda (t) + \mu  + \Delta )&\lesssim& P\left(\sum\limits_{k = 1}^{N(t)} {({X_k} + c)}  - \mu \lambda (t) \in x + \Delta \right)\nonumber\\
   &\lesssim& L_{F_\Delta}^2\lambda (t) F(x - c\lambda (t) + \mu  + \Delta )\label{C3101}
\eeq
holds uniformly for all $x \geq \gamma \lambda (t)$ as $t \to \infty $.
In particular, if $F(x + \Delta ) \in \mathcal{IR}$, then for 
any fixed $\gamma  > c$, the relation
\begin{equation}\label{C3102}
P \left(\sum\limits_{k = 1}^{N(t)} {({X_k} + c)}  - \mu \lambda (t) \in x + \Delta \right) \sim \lambda (t)F(x - c\lambda (t)+\mu + \Delta )
\end{equation}
holds uniformly for all $x \geq \gamma \lambda (t)$ as $t \to \infty $.
\end{cor}

 The second theorem gives precise local large deviation probabilities for  the claim surplus process of the generalized risk model introduced in section 1.
\begin{thm}\label{thm>thm32} 
In the general risk model introduced in section 1, suppose that $E{(X^ + )^r} < \infty $ for some $r > 1$ and $F_{\Delta}(x)=F(x + \Delta )$ is almost decreasing. In addition, assume that there exists a positive number $\nu$ such that
 \be\label{T3201}
  \limsup\limits_{t\to\infty}\frac{b (t)}{\lambda (t)}=\nu<\infty.\ee Assume that
\begin{equation}\label{T3202}
 \frac{Y(t)}{  b (t)} \mathop  \to \limits^P 1, ~t\to\infty.
\end{equation} and (\ref{T3101}) holds for some $p > |\beta ({F_\Delta })|$ and all $\delta  > 0$ as $t \to \infty $.
If $F(x + \Delta ) \in \mathcal{OR}$,
then for any fixed $\gamma >\nu$, the relation
\be \label{T3203}l_{F_\Delta}^3\lambda (t) F(x + \mu  + \Delta ) \lesssim P(S(t) - ES(t) \in x + \Delta ) \lesssim L_{F_\Delta}^3\lambda (t) F(x + \mu  + \Delta )\ee
holds for all $x \geq \gamma \lambda (t)$ as $t \to \infty $. In particular, if $F(x + \Delta ) \in \mathcal{IR}$, then for every fixed $\gamma  > \nu$, the relation
\begin{equation}\label{T3204}
 P(S(t) - ES(t) \in x + \Delta ) \sim \lambda (t)F(x +\mu+ \Delta )
\end{equation}
holds uniformly for all $x \geq \gamma \lambda (t)$ as $t \to \infty $.

\end{thm}

\section{Proofs of theorems and corollaries}
\setcounter{equation}{0}\setcounter{dei}{0}\setcounter{exa}{0}
\setcounter{thm}{0}\setcounter{cor}{0}\setcounter{lem}{0}\setcounter{pro}{0}\setcounter{rmk}{0}\mbox{}

Throughout this section, $C$ will represent a positive constant though its
value may change from one place to another. For $n\geq 1$, we denote by $S_n=\sum\limits_{k=1}^nX_k$ the $n$-th partial sum of a sequence
$\{X_k,k\geq 1\}$.

Before the proof of the main results, we first present several lemmas which will play important roles in the proofs of the theorems.
\begin{lem}\label{lem>lem401}
Let $\{ X, X_k,k \geq 1\} $ be a sequence of $i.i.d$ real-valued $r.v.s$ with a common distribution $F$. If $F(x + \Delta )$ is almost decreasing and $0 < \mu_+ = E(X^+ ) < \infty $, then for each fixed $v > 0$, there exist positive numbers $x_0$, $c_1$ and $c_2$ such that
\begin{equation}\label{l4101}
  P(S_n \in x + \Delta ) \leq c_1nF(vx + \Delta ) + c_2\left(\frac{n}{x}\right)^{\frac{1}{v}}
\end{equation}
holds for all $x \geq {x_0}$ and $n=1,2,\cdots.$
\end{lem}
\begin{proof}
We use arguments similar to those in the proof of Chen et al. (2011)~
 with some modifications.
For every fixed $v>0$, we denote ${\widetilde X_n} = {X_n}I({X_n} \leq vx)$ 
 and ${\widetilde S_n} = \sum\limits_{k = 1}^n {{{\widetilde X}_k}} $ for $n=1,2,\cdots$. Our starting point is the decomposition
\begin{align}
    P({S_n} \in x + \Delta ) = & P({S_n} \in x + \Delta ,\bigcap_{k=1}^n\{X_k \leq vx\}) + P({S_n} \in x + \Delta ,\bigcup_{k=1}^n\{X_k > vx\})\nonumber\\
    \hat{=} & {I_0}(x,n) + {I_1}(x,n). \label{pl4101}
\end{align}
First, we estimate ${I_0}(x,n)$: By Chebyshev's inequality, for a positive number $h=h(x,n)$ which
will be specified later, we have
\begin{align}
  {I_0}(x,n) = &P({\widetilde S_n} \in x + \Delta )\nonumber\\
    &\leq 
     {e^{ - hx}}{(E({e^{h{{\widetilde X}_1}}} - 1) + 1)^n}\nonumber\\
    &\leq {e^{ - hx}}\exp \left (n \int_0^{vx} {({e^{hy}} - 1} )F(dy)\right)\nonumber\\
    &\leq \exp \left (-hx+ \frac{n\mu_+(e^{hvx} - 1)}{{vx}}\right),  \label{pl4102}
\end{align}
where the last step is obtained by the monotonicity in $y \in (0, + \infty )$ of the function $\frac{{{e^{hy}} - 1}}{y}$.
Hence, by taking $h = \frac{1}{{vx}}\log (\frac{x}{{n{\mu_+ }}} + 1) > 0$ in (\ref{pl4102}), we obtain that
\begin{equation}\label{pl4103}
 {I_0}(x,n) \leq \exp \left( - \frac{1}{v}\log \left(\frac{x}{{n{\mu_+ }}} + 1\right) + \frac{1}{v}\right) \leq c_2\left(\frac{n}{x}\right)^{\frac{1}{v}},
\end{equation}
where  $c_2 = (\mu_+e)^{1/v}$. 
Next, we estimate ${I_1}(x,n)$: Since $F(x + \Delta )$ is almost decreasing, there exist positive numbers ${x_0}$ and ${c_1}$ such that
\begin{equation*}\label{eq2.11}
 \mathop {\sup }\limits_{z \geq vx} F(z + \Delta ) \leq {c_1} F(vx + \Delta )
\end{equation*}
holds for all $x \geq {x_0}$. Hence,
we have that
\begin{align}
 {I_1}(x,n) \leq & nP({S_n} \in x + \Delta ,{X_n} > vx)\nonumber \\
              = & n\int_{ - \infty }^{ + \infty } {P({X_n} \in x - y + \Delta ,{X_n} > vx)} P({S_{n - 1}} \in dy)\nonumber\\
               &\leq n\mathop {\sup }\limits_{z \geq vx} F(z + \Delta )\leq {c_1}nF(vx + \Delta ) \label{pl4104}
\end{align}
holds for all $x \geq {x_0}$. Substituting (\ref{pl4103}) and (\ref{pl4104}) into (\ref{pl4101}), we obtain (\ref{l4101}) immediately. This completes the proof of Lemma \ref{lem>lem401}.
 \end{proof}

The next lemma is a special case of Theorem 3.1 in Cheng and Li (2016):
\begin{lem}\label{lem>lem402}
Let $\{X, {X_k},k \geq 1\} $ be a sequence of $i.i.d$ real-valued $r.v.s$ with a common distribution $F$ of finite mean $\mu $. Suppose that $E{(X^+ )^r} < \infty $ for some $r>1$, $F_\Delta(x)=F(x + \Delta ) \in \mathcal{OR}$ and $F(x + \Delta )$ is almost decreasing, then for every fixed $\gamma  > 0$, it holds that
\begin{equation}\label{l4201}
 l_{F_{\Delta}} \leq \lim_{n\to\infty}\inf_{x\geq \gamma n}\frac{P(S_n-n\mu\in x+\Delta)}{nF(x+\mu+\Delta)} \leq \lim_{n\to\infty}\sup_{x\geq \gamma n}\frac{P(S_n-n\mu\in x+\Delta)}{nF(x+\mu+\Delta)}\leq  L_{F_{\Delta}} .
\end{equation}
\end{lem}

The following lemma can be found in Ng et al. (2003) and Chen et al. (2011):
\begin{lem}\label{lem>lem403}
Let $\{\xi_t, t\geq 0\}$ be a nonnegative stochastic process with $E\xi_t\to 1$ as $t\to\infty$.
Then, the following statements are equivalent:\\
\textup{(i)} $\xi_t\mathop\to\limits^p 1$ as $ t\to\infty$;\\
\textup{(ii)} $E\xi_tI(\xi_t >1+\varepsilon)\to 0$ as $t\to\infty$ for every fixed $\varepsilon > 0$; and\\
\textup{(iii)} $E\xi_tI(\xi_t\leq 1-\delta)\to 0$ as $t\to\infty$ for every fixed $\delta\in(0,1)$.
\end{lem}

We shall need the following lemma in the sequel:
\begin{lem}\label{lem>lem404}
Suppose that $f \in \mathcal{OR} $ and $f$ is almost decreasing. Then 
\begin{equation*} \label{l4401}{x^{ - p}} = o(f(x))
\end{equation*} holds for all $p > |\beta (f)|$ as $x\to\infty$.
\end{lem}
\begin{proof}

The proof of Lemma \ref{lem>lem404} is similar to that of Lemma 3.5 in Tang and  Tsitsiashvili (2003), so it is omitted.
\end{proof}
Now, we are in a position to prove Theorem \ref{thm>thm31}.

{\bf Proof of Theorem \ref{thm>thm31}.}

For an arbitrarily fixed number $\delta \in(0, \min\{\frac{\gamma }{|\mu |},1\})$ where $\frac{1}{0}=\infty$ by convention,
we 
divide $P({S_{N(t)}} - \mu \lambda (t) \in x + \Delta )$ into three parts as
\begin{align}
&P({S_{N(t)}} - \mu \lambda (t) \in x + \Delta )\nonumber\\
 = &\left (\sum\limits_{n < (1 - \delta )\lambda (t)}  + \sum\limits_{|n - \lambda (t)| \leq \delta \lambda (t)}  + \sum\limits_{n > (1 + \delta )\lambda (t)} \right )P({S_n}{\rm{ - }}\mu \lambda (t) \in x + \Delta )P(N(t) = n) \nonumber \\
\hat  = & {J_1}(x,t) + {J_2}(x,t) + {J_3}(x,t).\label{PT3101}
\end{align}
  We will estimate $J_i(x,t),i=1,2,3$, respectively. First, we will prove that
 \be\label{PT3102}
  J_1(x,t)=o(\lambda(t)F(x+\mu+\Delta))
 \ee
holds uniformly for all $x\geq \gamma t$ as $t\to\infty$.

 We will consider two scenarios according to $\mu\geq 0$ and $\mu<0$, respectively.

{\it Scenario 1}: $\mu \geq 0$. By Lemma \ref{lem>lem402}, there exists a positive integer ${n_0}$ such that
\begin{equation}\label{PT3103}
P(S_n-n\mu\in y+\Delta)\leq (1 + \delta )nL_{F_\Delta} F(y + \mu  + \Delta )
\end{equation}
holds for all $y\geq \frac{\gamma+\delta\mu }{{1 - \delta }}n$ and $n\geq n_0$. For this fixed $n_0$ and sufficiently large $t$, we divide $J_1(x,t)$ into two parts as
\beq
 J_1(x,t)&=&\left(\sum_{n=1}^{n_0}+\sum_{n_0<n<(1 - \delta )\lambda (t)}
 \right)
 P(S_n \in   x +\mu \lambda (t)+ \Delta )P(N(t) = n)\nonumber\\
 &=&J_{11}(x,t)+J_{12}(x,t). \label{PT3104}
\eeq
Note that $x>\gamma\lambda(t)$ and $n<(1 - \delta )\lambda (t)$ imply that
    $x + \mu \lambda (t) - n\mu  \geq x+\delta\mu\lambda(t) \geq \frac{\gamma+\delta\mu }{{1 - \delta }}n.$
Hence, by taking $y=x+\mu\lambda(t)-n\mu$ in (\ref{PT3103}), it follows that
\beq
J_{12}(x,t)&\leq& \sum_{n_0<n<(1 - \delta )\lambda (t)}
{(1 + \delta )nL_{F_\Delta} F(x + \mu \lambda (t) - n\mu  + \mu  + \Delta )} P(N(t) = n)\nonumber\\
           &\leq& L_{F_\Delta} (1 +\delta )\mathop {\sup }\limits_{z \geq x + \mu } F(z + \Delta )EN(t)I(N(t) < (1 - \delta )\lambda (t)).\label{PT3105}
\eeq
Note that $E(X^+)^r<\infty$ for $r>1$ implies $\beta(F_{\Delta})<-1$, it follows from (\ref{T3101}) that
\begin{equation*}\label{PT3106}
 EN{(t)}I(N(t) > (1 + \varepsilon )\lambda (t)) = o(\lambda (t))
\end{equation*}
holds for all $\varepsilon\in(0,1)$ as $t\to\infty$, which yields from Lemma \ref{lem>lem403} that
\begin{equation}\label{PTB3101}
 EN(t)I(N(t)< (1 - \delta )\lambda (t)) = o(\lambda (t)).
\end{equation}
In addition, since $F(x+\Delta)$ is almost decreasing, we have that
\be\label{PT3107}
\sup \limits_{z \geq y } F(z + \Delta )\leq CF(y+\Delta)\ee
holds for sufficiently large $y$. This, along with
(\ref{PT3105}) and (\ref{PTB3101}) implies that
\be \label{PT3108}
\limsup_{t\to\infty}\sup\limits_{x \geq \gamma \lambda (t)} \frac{J_{12}(x,t)}{\lambda(t)F(x+\mu+\Delta)}=0.
\ee
On the other hand, taking $v = \frac{1}{p}$ in Lemma \ref{lem>lem401} with $p > |\beta ({F_\Delta })| \geq 1$, there exist positive numbers ${c_1}$ and ${c_2}$ such that
\beq
  &&{J_{11}}(x,t) \leq \sum\limits_{n = 1}^{{n_0}} {\left\{{c_1}nF\left(\frac{{x + \mu \lambda (t)}}{p} + \Delta \right) + {c_2}{{ \left(\frac{n}{{x + \mu \lambda (t)}}\right )}^p}\right \}} P(N(t) = n)\nonumber\\
 &\leq &\left\{{c_1}{n_0}F\left (\frac{{x + \mu \lambda (t)}}{p} + \Delta\right ) + {c_2} \left(\frac{n_0}{{x + \mu \lambda (t)}}\right )^p\right\}P(N(t)<(1-\delta)\lambda(t)).\label{PT3109}
  \eeq
Note that $F(x + \Delta ) \in \mathcal{OR}$ and (\ref{PT3107}) yield that
\be\label{PT3110}
F\left (\frac{{x + \mu \lambda (t)}}{p} + \Delta\right ) \leq \sup_{z>\frac{x+\mu}{p}}F(z + \Delta ) \leq CF(x + \mu  + \Delta )
\ee
holds for sufficiently large $x$ and $t$. Furthermore, by Lemma \ref{lem>lem404}, we have that
\be\label{PT3111}
(x + \mu \lambda (t))^{ - p} \leq (x+\mu)^{ - p} =  o(F(x + \mu  + \Delta )),
\ee
holds uniformly for all $x\geq \gamma\lambda(t)$ as $t\to\infty.$
Thus, from  (\ref{PTB3101}) and (\ref{PT3109})-(\ref{PT3111}), we have
\be \label{PT31110}
\limsup_{t\to\infty}\sup\limits_{x \geq \gamma \lambda (t)} \frac{J_{11}(x,t)}{\lambda(t)F(x+\mu+\Delta)}=0.
\ee
 Obviously,
from (\ref{PT3108}) and (\ref{PT31110}), it follows that (\ref{PT3102}) 
 holds uniformly for all $x\geq \gamma\lambda(t)$ as $t\to\infty.$

{\it Scenario 2}: $\mu <0 $. 
we take $\widetilde \gamma  = \max \{ |\mu | + 1,\gamma \} $ and split the $x-$region into two disjoint regions as $$[\gamma \lambda (t), + \infty ) = [\gamma \lambda (t),\widetilde \gamma \lambda (t)) \cup [\widetilde \gamma \lambda (t), + \infty ).$$
For the first $x$-region, 
 note that
\ben
 \sup_{\gamma \lambda (t) \leq x < \widetilde \gamma \lambda (t)}F(\lambda(t)+\Delta)\leq \sup_{z\geq \frac{x+\mu}{\widetilde{\gamma}}}F(z+\Delta)\leq CF\left (\frac{x+\mu}{\widetilde{\gamma}}+\Delta \right )
\een
holds for sufficient large $x$ and $t$.
It follows from $F(x + \Delta ) \in \mathcal{OR}$ and (\ref{T3102}) that 
\be\label{PT3112}
\limsup_{t\to\infty}\sup_{\gamma \lambda (t) \leq x < \widetilde \gamma \lambda (t)} \frac{J_1(x,t)}{\lambda (t)F(x+\mu+\Delta)}
\leq \limsup_{t\to\infty}\sup_{\gamma \lambda (t) \leq x < \widetilde \gamma \lambda (t)} \frac{P(N(t)<(1-\delta)\lambda(t))}{ \lambda (t)F(x + \mu  + \Delta )}=0.
\ee
For the second $x$-region $x \geq \widetilde \gamma \lambda (t)$, it follows from $x>\gamma\lambda(t)$ and $n<(1 - \delta )\lambda (t)$ that
\ben
  x + \mu \lambda (t) - n\mu  \geq x + \mu \lambda (t) \geq \frac{{\widetilde \gamma  + \mu }}{{1 - \delta }}n.
\een
Using a method similar to the previous scenario, we can easily obtain that (\ref{PT3102}) holds uniformly for all $x\geq \widetilde{\gamma}\lambda(t)$ as $t\to\infty.$
Combining with (\ref{PT3112}) we obtain that (\ref{PT3102}) holds uniformly also for all $x\geq \gamma\lambda(t)$ as $t\to\infty.$

Next, we will prove that
 \be\label{PT3113}
  \limsup_{t\to\infty}\sup\limits_{x \geq \gamma \lambda (t)} \frac{J_3(x,t)}{\lambda(t)F(x+\mu+\Delta)}=0.
 \ee
 We will consider two cases according to $\mu\geq 0$ and $\mu<0$ also.  

{\it Case 1}: $\mu \geq 0$. Taking $v = \frac{1}{p}<1$ in Lemma \ref{lem>lem401} again, 
 it follows that
\begin{align*}
{J_3}(x,t) \le& \sum\limits_{n > (1 + \delta )\lambda (t)}^{} {\left [{c_1}nF\left(\frac{{x + \mu \lambda (t)}}{p} + \Delta \right) + {c_2}{{\left(\frac{n}{{x + \mu \lambda (t)}}\right )}^p}\right]} P(N(t) = n)\\
 =& {c_1}F\left(\frac{{x + \mu \lambda (t)}}{p} + \Delta\right )EN(t)I(N(t) > (1 + \delta )\lambda (t))\\
 &+ {c_2}{(x + \mu \lambda (t))^{ - p}}EN^p(t)I(N(t) > (1 + \delta )\lambda (t)).
\end{align*}
Combining with (\ref{T3101}), (\ref{PT3110}) and (\ref{PT3111}), we obtain that (\ref{PT3113}) holds.

{\it Case 2}: $\mu <0$. Denote
 \ben
 \gamma '=\left\{\begin{array}{cc}
                       |\mu|,&\textup{~if~} \gamma  + \mu  \geq 0\\
                       |\mu | - \frac{|\gamma  + \mu |}{1 + \delta },&\textup{~if~} \gamma  + \mu  < 0
                 \end{array}
           \right ..
\een
Since $x>\gamma\lambda(t)$ and $n>(1+ \delta )\lambda (t)$ imply that
   $ x + \mu \lambda (t) - n\mu  \geq \gamma' n,$
 by Lemma \ref{lem>lem402}, we have that (\ref{PT3103}) holds for all
 $y>\gamma' n$ when $n$ is sufficiently large, 
  which yields that
\begin{align*}
J_3(x,t) &\leq \sum\limits_{n > (1 + \delta )\lambda (t)}^{} {(1 + \delta )}n L_{F_\Delta} F(x + \mu \lambda (t) - n\mu  + \mu  + \Delta )P(N(t) = n)\\
 &\leq (1 + \delta )L_{F_\Delta}\mathop {\sup }\limits_{z \geq x} F(z + \mu  + \Delta )EN(t)I(N(t) > (1 + \delta )\lambda (t)).
\end{align*}
Hence, combining with Lemma \ref{lem>lem403}, (\ref{PT3113}) follows from (\ref{PT3107}) and (\ref{T3102}).

Finally, we estimate $J_2(x,t)$: Recall that $\delta \in(0, \min\{\frac{\gamma }{|\mu |},1\})$.  It follows from $x>\gamma\lambda(t)$ and $|n - \lambda (t)| \leq \delta \lambda (t)$ that
 \ben
   x  + \mu \lambda (t) - n\mu  \geq x - |\mu |\delta \lambda (t) \geq \frac{{\gamma  - |\mu |\delta }}{{1 + \delta }}n.
 \een
 For sufficiently large $t$, it follows from Lemma \ref{lem>lem402} that
\begin{align*}
&(1-\delta)nl_{F_\Delta}  F(x + \mu \lambda (t) - n\mu  + \mu  + \Delta )\\
 &\leq P({S_n} - n\mu  \in x + \mu \lambda (t) - n\mu  + \Delta ) \\
 &\leq (1+\delta)nL_{F_\Delta}  F(x + \mu \lambda (t) - n\mu  + \mu  + \Delta )
\end{align*}
and
\ben
 |\mu \lambda (t) - n\mu | \leq \frac{{|\mu |\delta }}{\gamma }x \leq \frac{{2|\mu |\delta }}{\gamma }(x + \mu )
\een
hold for all $x>\gamma\lambda(t)$ and $|n - \lambda (t)| \leq \delta \lambda (t)$.
Hence, for sufficiently large $t$, it holds that
\beq
&&  (1-\delta)^2l_{F_\Delta}\lambda(t)\inf\limits_{1-\frac{2|\mu |\delta }{\gamma }<y<1+\frac{2|\mu |\delta }{\gamma }}F((x+\mu)y+\Delta)P(|N(t) - \lambda (t)| \leq \delta \lambda (t))\nonumber\\
&\leq & (1-\delta)l_{F_\Delta} \sum\limits_{|n - \lambda (t)| \leq \delta \lambda (t)}n F(x + \mu \lambda (t) - n\mu  + \mu  + \Delta )P(N(t)=n)\nonumber\\
&\leq & \inf\limits_{x \geq \gamma \lambda (t)} J_2(x,t)\leq \sup\limits_{x \geq \gamma \lambda (t)} J_2(x,t)\nonumber\\
&\leq & (1+\delta)L_{F_\Delta} \sum\limits_{|n - \lambda (t)| \leq \delta \lambda (t)}n F(x + \mu \lambda (t) - n\mu  + \mu  + \Delta )P(N(t)=n)\nonumber\\
&\leq &  (1+\delta)^2L_{F_\Delta}\lambda(t)\sup\limits_{1-\frac{2|\mu |\delta }{\gamma }<y<1+\frac{2|\mu |\delta }{\gamma }}F((x+\mu)y+\Delta)P(|N(t) - \lambda (t)| \leq \delta \lambda (t)).\nonumber\\\label{PT3114}
\eeq
 Using Lemma \ref{lem>lem403}, either (\ref{T3101}) or (\ref{T3102}) implies that $\frac{{N(t)}}{{\lambda (t)}}\mathop  \to \limits^p 1 \text{ as } t \to \infty$, which yields that
\begin{equation}\label{PT3115}
\lim_{t\to\infty}P(|N(t) - \lambda (t)| \leq \delta \lambda (t))=1.
\end{equation}
Hence, it follows from (\ref{PT3114}) and (\ref{PT3115}) that
\beqn
l_{F_\Delta}^2
\leq  \lim_{\delta\downarrow0}\liminf_{t\to\infty}\inf\limits_{x \geq \gamma \lambda (t)} \frac{ J_2(x,t)}{\lambda(t)F(x+\mu+\Delta)}
\leq  \lim_{\delta\downarrow0}\limsup_{t\to\infty}\sup\limits_{x \geq \gamma \lambda (t)} \frac{ J_2(x,t)}{\lambda(t)F(x+\mu+\Delta)}\leq L_{F_\Delta}^2.
\eeqn
 Combining with (\ref{PT3101}), (\ref{PT3102}) and (\ref{PT3113}), we obtain that (\ref{T3103}) holds uniformly for all $x \geq \gamma \lambda (t)$ as $t \to \infty $ immediately.  This completes the proof of the first part of Theorem \ref{thm>thm31}. The second part of Theorem \ref{thm>thm31} immediately follows from the first part since $F(x+\Delta)\in \cal{IR}$ implies that $l_{F_{\Delta}}=L_{F_{\Delta}}=1.$

{\sl  Proof of Corollary \ref{cor>cor31}.}
Let ${\hat X_k} = {X_k} + c$ for $k = 1,2,3\cdots$ and ${\hat S_{N(t)}} = \sum\limits_{k = 1}^{N(t)} {{{\hat X}_k}}  = \sum\limits_{k = 1}^{N(t)} {({X_k} + c)}$. 
 Let $\hat F$ be the common distribution of $\{ {\hat X_k},k \geq 1\} $. It is easy to prove that $F(x+\Delta)\in \mathcal{L}\cap\mathcal{OR}$ implies that $\hat F(x + \Delta )\in \mathcal{OR}$ and $L_{F_\Delta }   = L_{{\hat F}_\Delta } $ and $l_{F_\Delta }   = l_{{\hat F}_\Delta } $.
Hence, by Theorem \ref{thm>thm31}, for any $\gamma>0$, 
 the relation
\begin{equation*}\label{eq3.8}
  \lambda (t)l_{F_\Delta}^2\hat F(\hat x + \mu  + c + \Delta ) \lesssim P({\hat S_{N(t)}} - (\mu  + c)\lambda (t) \in \hat x + \Delta ) \lesssim \lambda (t)L_{F_\Delta}^2\hat F(\hat x + \mu  + c + \Delta )
\end{equation*}
holds uniformly for all $\hat x \geq \hat \gamma \lambda (t)$ as $t \to \infty $, where $\hat{\gamma}=\gamma-c > 0$ and $\hat{x}=x-c\lambda(t)>\hat{\gamma}\lambda(t)$.
 Hence, we obtain that relation (\ref{C3101}) 
holds uniformly for all $x \geq \gamma \lambda (t)$ as $t \to \infty $.
In particular, when $F(x + \Delta ) \in \mathcal{IR}$, (\ref{C3102}) holds by Proposition \ref{pro2.3}.

Now we stand on the position to prove Theorem \ref{thm>thm32}.

{\it Proof of Theorem \ref{thm>thm32}.}

For an arbitrarily fixed $0 < \delta  < 1$, we divide $P(S(t) - ES(t) \in x + \Delta )$ into three parts as
\begin{align}
&P(S(t) - ES(t) \in x + \Delta )\nonumber\\
= & \left (\int_0^{ (1 - \delta ) b (t)}  + \int_{(1- \delta)  b (t)}^{(1+ \delta)  b (t)}  + \int_{(1 + \delta ) b (t)}^{\infty} \right )P\left (\sum\limits_{i = 1}^{N(t)} {X_i}  - \mu \lambda (t) \in x + y + b (t) + \Delta \right )dP(Y(t) \leq y)\nonumber \\
\widehat  =& {J_1}(x,t) + {J_2}(x,t) + {J_3}(x,t).\label{PT3201}
\end{align}
We will estimate $J_i(x,t),~i=1,2,3$, respectively. First, we estimate ${J_1}(x,t)$:
  By (\ref{T3201}), there exists a number $w\in (\nu, \gamma)$ such that
\be\label{PT3202}
b(t)\leq w\lambda(t)
\ee
holds for sufficiently large $t$. Since $x>\gamma \lambda(t)$ and $0\leq y\leq(1 - \delta ) b (t)$ imply that
$x + y -  b (t) \geq x -  b (t) \geq (\gamma  - w)\lambda (t)
$ 
 and
$x + y -  b (t)+\mu \geq x-  w\lambda(t)+\mu \geq (1-\frac{  w}{\gamma})(x+\mu)$
hold for sufficiently large $t$ and $x>\gamma t$, it follows from Theorem \ref{thm>thm31} that 
\begin{align*}
{J_1}(x,t) &\leq \int_0^{ (1 - \delta ) b (t)}(1 + \delta )\lambda (t)L_{F_\Delta}^2 F(x + y -  b (t) + \mu  + \Delta )dP(Y(t)\leq y)\nonumber\\
 &\leq (1 + \delta )\lambda (t)L_{F_\Delta}^2 \mathop {\sup }\limits_{z \geq 1 - \frac{{w}}{\gamma }} F((x + \mu)z  + \Delta )P(Y(t) \leq (1 - \delta ) b (t))\nonumber
\end{align*}
holds for sufficiently large $t$ and $x>\gamma t$.
Furthermore, since $F(x + \Delta )$ is almost decreasing, it follows from (\ref{T3202}) and $F(x + \Delta ) \in \mathcal{OR}$ that (\ref{PT3102}) holds uniformly for all $x\geq \gamma\lambda(t)$ as $t\to\infty$.

Next, we estimate ${J_3}(x,t)$:
 It should be noted that $x>\gamma \lambda(t)$ and $y>(1 + \delta ) b (t)$ imply that
$x + y -  b (t) \geq x \geq \gamma \lambda (t)$
and
$x + y -  b (t)+\mu \geq x +\mu$  
hold for sufficiently large $t$. Hence, we obtain from Theorem \ref{thm>thm31} that
\begin{align*}
{J_3}(x,t) &\leq \int _0^{ (1 + \delta ) b (t)} {(1 + \delta )\lambda (t)L_{F_\Delta}^2 F(x + y -  b (t) + \mu  + \Delta )} dP(Y(t)\leq y)\nonumber\\
 &\leq {(1 + \delta )\lambda (t)L_{F_\Delta}^2\sup_{z>x+\mu} F(z  + \Delta )}P(Y(t)>(1 + \delta ) b (t))\nonumber
\end{align*}
hold for sufficiently large $t$ and $x>\gamma t$. On the other hand,
since $F(x + \Delta )$ 
 is almost decreasing, it follows (\ref{T3202}) that (\ref{PT3113}) holds.

Finally, we estimate ${J_2}(x,t)$:  Note that $x>\gamma \lambda(t)$ and $(1 - \delta ) b (t)<y<(1 + \delta ) b (t)$ imply that
 $x + y -  b (t) \geq x - \delta b (t) \geq (\gamma  - \delta w)\lambda (t)$
 and
$x + y -  b (t)+\mu\in[ (1-\frac{\delta w}{\gamma})(x+\mu),~(1+\frac{\delta w}{\gamma})(x+\mu)]$
 hold for sufficiently large $t$ and $x>\gamma t$. Hence, we obtain from Theorem \ref{thm>thm31} that
\begin{align}
 &(1 - \delta )\lambda (t)l_{F_\Delta}^2 \mathop {\inf }\limits_{(1 - \frac{{\delta w}}{\gamma } ) \leq z \leq (1 + \frac{{\delta w}}{\gamma }  )} F((x + \mu)z  + \Delta )P(|Y(t) -  b (t)| \leq \delta  b (t))\nonumber\\
 & \leq \int_{(1-\delta)b(t)}^{(1+ \delta)b(t)}{(1 - \delta )\lambda (t)l_{F_\Delta}^2F(x + y-  b (t) + \mu  + \Delta )}d P(Y(t) \leq y)\nonumber\\
&\leq {J_2}(x,t)\nonumber\\
 & \leq \int_{(1- \delta)  b (t)}^{(1+ \delta)  b (t)} {(1 + \delta )\lambda (t)L_{F_\Delta}^2F(x + y-  b (t) + \mu  + \Delta )}d P(Y(t) \leq y)\nonumber\\
 &\leq (1 + \delta )\lambda (t)L_{F_\Delta}^2 \mathop {\sup }\limits_{(1 - \frac{{\delta w}}{\gamma } ) \leq z \leq (1 + \frac{{\delta w}}{\gamma }  )} F((x + \mu)z  + \Delta )P(|Y(t) -  b (t)| \leq \delta  b (t))\label{PT3203}
\end{align}
holds for sufficiently large $t$ and $x>\gamma t$, which yields from (\ref{T3202}), (\ref{PT3102}), (\ref{PT3113}), (\ref{PT3201}) and (\ref{PT3203}) that
\beqn
 &&(1 - \delta )l_{F_\Delta}^2 \liminf_{x\to\infty}\mathop {\inf }\limits_{(1 - \frac{{\delta w}}{\gamma } ) \leq z \leq (1 + \frac{{\delta w}}{\gamma }  )} \frac{F((x + \mu)z  + \Delta )}{F(x+\mu+\Delta)}\nonumber\\
&\leq& \liminf_{t\to\infty}\inf_{x>\gamma t}\frac{P(S(t) - ES(t) \in x + \Delta )}{\lambda(t)F(x+\mu+\Delta)}\leq \limsup_{t\to\infty}\sup_{x>\gamma t}\frac{P(S(t) - ES(t) \in x + \Delta )}{\lambda(t)F(x+\mu+\Delta)}\nonumber\\
 & \leq &(1 +\delta )L_{F_\Delta}^2 \limsup_{x\to\infty}\mathop {\sup }\limits_{(1 - \frac{{\delta w}}{\gamma } ) \leq z \leq (1 + \frac{{\delta w}}{\gamma }  )} \frac{F((x + \mu)z  + \Delta )}{F(x+\mu+\Delta)}
\eeqn
holds for sufficiently large $t$. 
By the arbitrariness of $\delta$ and the definitions of $L_{F_\Delta}$ and $l_{F_\Delta}$, we obtain that (\ref{T3203}) holds uniformly for all $x\geq \gamma\lambda(t)$ as $t\to\infty$.  This completes the proof of Theorem \ref{thm>thm32}.



\end{document}